\documentclass[10pt]{article}

\usepackage{anyfontsize}

\usepackage{amsmath}
\usepackage{amsthm}

\usepackage{tikz}

\usepackage[bookmarks]{hyperref}
\hypersetup{
        colorlinks=true,
        linkcolor=black,
        anchorcolor=black,
        citecolor=black,
        urlcolor=black,
        pdftitle={Counting parity palindrome compositions},
        pdfauthor={Vincent Vatter},
}

\makeatletter

\theoremstyle{plain}
\newtheorem*{theorem}{Theorem}

\title{\vspace{-1in}\sc Counting Parity Palindrome Compositions}
\author{%
Vincent Vatter\\[-0.25ex]
\small Department of Mathematics\\[-0.5ex]
\small University of Florida\\[-0.5ex]
\small Gainesville, Florida USA\\[-1.5ex]
}

\vspace{-0.25in}

\date{}

\begin{document}
\maketitle
\thispagestyle{empty}


\noindent A \emph{composition} of the positive integer $n$ is an ordered sequence of positive integers (its \emph{parts}) that sum to $n$. A \emph{parity palindrome composition} (\emph{ppc}) is a composition that reads the same both forward and backward when reduced modulo~$2$; for example, $32141$ is a ppc of $11$.
(For brevity, we do not write a $+$ between the parts of our compositions.)
Andrews and Simay~\cite{andrews:parity-palindro:} showed that ppcs have a surprisingly simple formula; see also Just~\cite{just:compositions-th:}. We give a recursive proof.

\begin{theorem}
For all $n\ge 1$, the number of ppcs of $2n$ \emph{(}resp., $2n+1$\emph{)} is $2\cdot 3^{n-1}$.
\end{theorem}

\begin{proof}
We say that a ppc of $2n$ (resp., $2n+1$) \emph{produces} two to four ppcs of $2n+2$ (resp., $2n+3$) by the following rules:
\begin{enumerate}
	\item[(A)] introduce two new parts equal to $1$, one on each side of the ppc;
	\item[(B)] add $1$ to both the first and last parts, or if it has only one part, add $2$ to it;
	\item[(C$_1$)] if the first part is $1$, add $2$ to the last part;
	\item[(C$_2$)] if the last part is $1$, add $2$ to the first part.
\end{enumerate}

The ppcs of $2$, $3$, $4$, and $5$ are shown below, with lines to indicate productions. This proves the theorem for $n=1$ and $n=2$.
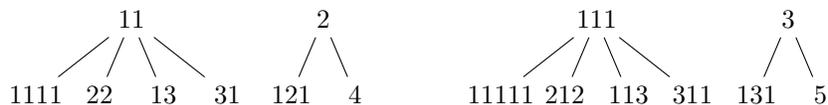
\begin{center}
\begin{tikzpicture}[xscale=0.85]
\node (11) at (2.5,1) {$11$};
\node (1111) at (1,0) {$1111$};
\node (22) at (2,0) {$22$};
\node (13) at (3,0) {$13$};
\node (31) at (4,0) {$31$};
\draw (11)--(1111); 
\draw (11)--(22); 
\draw (11)--(13); 
\draw (11)--(31); 

\node (2) at (5.5,1) {$2$};
\node (121) at (5,0) {$121$};
\node (4) at (6,0) {$4$};
\draw (2)--(121);
\draw (2)--(4);
\end{tikzpicture}
\quad
\quad
\quad
\begin{tikzpicture}[xscale=0.85]
\node (111) at (8.5,1) {$111$};
\node (11111) at (7,0) {$11111$};
\node (212) at (8,0) {$212$};
\node (113) at (9,0) {$113$};
\node (311) at (10,0) {$311$};
\draw (111)--(11111);
\draw (111)--(212);
\draw (111)--(113);
\draw (111)--(311);

\node (3) at (11.5,1) {$3$};
\node (131) at (11,0) {$131$};
\node (5) at (12,0) {$5$};
\draw (3)--(131);
\draw (3)--(5);
\end{tikzpicture}
\end{center}

To verify that every ppc is produced precisely once, we split them into types:
\begin{enumerate}
	\item[(A)] if the first and last part are both $1$, it was produced by rule (A);
	\item[(B)] if the first and last parts are both at least $2$, it was produced by rule (B);
	\item[(C)] if the first part is $1$ and the last part is at least $2$ (or vice versa), then the last part must actually be at least $3$ by the parity palindrome constraint, and the ppc was produced by rule (C$_1$) (or (C$_2$) in the vice versa case).
\end{enumerate}

One can use the picture above to check that for ppcs of $4$ and $5$, precisely one-third are of each of the types (A), (B), and (C).
Those of type (A) produce $4$ ppcs, one each of types (A) and (B) and two of type (C); those of type (B) produce $2$ ppcs, of types (A) and (B); and those of type (C) produce $3$ ppcs, of types (A), (B), and (C). This ensures both that the number of ppcs is tripled at each stage and that at every stage precisely one-third of the ppcs are of each type, proving the theorem for all $n\ge 3$.
\end{proof}

\end{document}